\documentclass{mathscan}

\usepackage{amsmath,amscd,amssymb,amsthm,amstext,mathrsfs,dcolumn,
booktabs,enumerate}

\usepackage{pinlabel}

\usepackage{ifpdf}
\ifpdf
\usepackage[all,line,cmtip]{xy}      
\else
\usepackage[all,line,cmtip,dvips]{xy}
\fi

\CompileMatrices

\title{On vector bundles for a Morse decomposition of $L(\CP^n)$}
\author{Iver Ottosen}

\address{Department of Mathematical Sciences\\
Aalborg University Copenhagen\\
A.C. Meyers V\ae nge 15\\
2450 K\o benhavn SV\\
Denmark}

\email{ottosen@math.aau.dk}

\input{xypic}

\theoremstyle{plain}
   \newtheorem{theorem}{Theorem}[section]
   \newtheorem*{theorem*}{Theorem}
   \newtheorem{proposition}[theorem]{Proposition}
   \newtheorem{lemma}[theorem]{Lemma}
   \newtheorem{corollary}[theorem]{Corollary}
   \theoremstyle{definition}
   
   \newtheorem{definition}[theorem]{Definition}
   
   \theoremstyle{remark}
   
   \newtheorem{remark}[theorem]{Remark}

\newcommand{\NN}{{\mathbb N}}
\newcommand{\ZZ}{{\mathbb Z}}
\newcommand{\FF}{{\mathbb F}}

\newcommand{\RR}{{\mathbb R}}
\newcommand{\CC}{{\mathbb C}}
\newcommand{\HH}{{\mathbb H}}
\newcommand{\OO}{{\mathbb O}}

\newcommand{\TT}{{\mathbb T}}
\newcommand{\PP}{{\mathbb P}}

\newcommand{\CP}{{\mathbb{C}\mathrm{P}}}
\newcommand{\HP}{{\mathbb{H}\mathrm{P}}}
\newcommand{\CaP}{{\mathbb{O}\mathrm{P}}}

\newcommand{\Gr}[2]{{{\mathbf G}_{#1}(\CC^{#2})}}
\newcommand{\St}[2]{{{\mathbf V}_{#1}(\CC^{#2})}}
\newcommand{\Sta}[3]{{{\mathbf V}_{#1 , #3}(\CC^{#2})}}
\newcommand{\PStnull}[2]{{{\mathbf{PV}}_{#1}(\CC^{#2})}}
\newcommand{\PSt}[3]{{{\mathbf{PV}}_{#1 , #3}(\CC^{#2})}}
\newcommand{\PV}{{\mathbf{PV}}}
\newcommand{\B}[2]{{\mathbf{B}_{#1} (\CP^{#2} )}}
\newcommand{\hp}{{\rho}}
\newcommand{\fs}{{d\rho}}

\newcommand{\FC}{{\mathcal F}}
\newcommand{\Th}{{\mathrm{Th}}}

\DeclareMathOperator{\linspan}{span}
\DeclareMathOperator{\real}{Re}

\newcommand{\inp}[2]{{\langle #1, #2 \rangle}}
\newcommand{\inpl}[2]{{\langle \langle #1, #2 \rangle \rangle}}

\begin{document}

\begin{abstract}
We give a description of the negative bundles for the energy integral on the free loop space $L\CP^n$ in terms of
circle vector bundles over projective Stiefel manifolds. We compute the mod $p$ Chern classes of 
the associated homotopy orbit bundles.
\end{abstract}

\maketitle

\section{Introduction}
\label{sec:Intro}

This paper is a part of a program to study the homotopy type of the free loop space of a smooth manifold $M$.  
Our main interest is to understand the $\TT=S^1$-equivariant homotopy type. More precisely, we try to get information 
about the mod $p$ equivarant cohomology as a module over the Steenrod algebra.

We remark that this module is closely related to the cohomology of the topological cyclic homology spectrum $TC(M , p)$ \cite{BHM}. The topological cyclic homology spectrum is in turn an approximaton to the algebraic $K$-theory of $M$.

A general strategy for this is to equip the manifold with a Riemannian metric and consider the Morse theory of the energy functional $E$ defined by this metric. Since the energy is invariant under rotation of the loops, this captures not just the 
ordinary homotopy type of the loop space, but also the equivariant homotopy type.    

We focus on a very special case, namely the free loop space on a complex projective space. We choose the Riemannian metric to be the usual (Fubini-Study) metric. We consider this as a special case which might throw light on the general situation. 

However, another motivation for examining this special case closely comes from the unsolved closed geodesics problem:
Does any Riemannian metric on a compact simply connected smooth manifold $M$ of dimension greater than one 
admit infinitely many geometrically distinct closed geodesics? The answer is affirmative if the rational cohomology ring
of $M$ requires at least two generators (Vigu\' e-Puirrier \& Sullivan, Gromoll \& Meyer) or if $M$ is a globally 
symmetric space of rank larger than one (Ziller). It is also affirmative for 
the 2-sphere (Bangert, Franks, Angenent, Hingston). The most prominent examples
where the answer is not known are the spheres $S^m$, $m\geq 3$ together with the projective spaces 
$\CP^n$ (for $n\geq 2$), $\HP^n$ and Cayley's projective plane ${\CaP}^2$.

In this game, Morse theory of the energy integral on the free loop space $LM$ plays a central role. Therefore 
it is interesting to gather as much information as possible on the bundles controlling the Morse decomposition.

In \cite{KlingProj} Klingenberg studies the non-equivariant Morse theory of the free loop spaces on a projective space $LP^n$. Complex and quaternionic projective spaces as well as the Cayley projective plane are considered. Critical points for the energy integral are closed geodesics of various energy levels $0=e_0<e_1<\dots$. Those of energy level 
$e_q$ form a finite dimensional critical submanifold $B_q$ of $LP^n$.  There is a so-called negative vector bundle 
$\mu_q^-$ over $B_q$ which is essentially the tangent space of the unstable manifold given by exiting negative gradient trajectories. The energy levels also give a filtration of the free loop space $\FC (e_q)=E^{-1}([0,e_q])$.
Morse theory in this setting states that $\FC (e_q)$ is essentially obtained by attaching to
$\FC (e_{q-1})$ the disc bundle of $\mu_q^-$. One of the results in Klingenberg's article is a concrete calculation of
the negative bundles.

By the invariance of the energy functional the filtration is an equivariant filtration. The negative bundles will be 
$\TT$-equivariant bundles, so that they induce vector bundles on the Borel construction on $B_q$. 
We obtain a filtration of the Borel construction $ES^1\times_{S^1} LP^n$.  
The filtration quotients are the Thom spaces of these homotopy orbit bundles over $ES^1\times_{S^1} B_q$. 

The purpose of this paper is firstly to give a simpler description of the negative bundles for the complex projective
spaces as $\TT$-vector bundles over projective Stiefel manifolds (Theorem \ref{main1} and Definition \ref{def:nu}). Secondly, we calculate the mod $p$ Chern classes of the associated homotopy orbit bundles (Theorem \ref{main2}). This determines the action of the Steenrod algebra on the corresponding Thom spaces.
 
These results are partly motivated by \cite{BO5} where we compute the mod $p$ equivariant cohomology of $L\CP^n$ with respect to the action of the circle group $\TT$. The calculation uses the spectral
sequence coming from the energy filtration. This is a spectral sequence of modules over the Steenrod algebra. The computations in the present paper determines this action on the first page of the spectral sequence, and our hope is that this can lead to a computation of the Steenrod algebra action on $H^*_\TT (L\CP^n ; \FF_p )$.  

There is an alternative way of computing equivariant cohomology of $L\CP^n$. This uses the formality of the homotopy type of $\CP^n$ together with computations in cyclic homology. The method is described in \cite{NEH}. At the moment, it does not seem clear how to obtain the action of the Steenrod algebra from this method. However, there is no reason to believe that it is inherently impossible to do this, and our computation might very well help in understanding the relation between cyclic homology and cohomology operations.

\section{Morse theory for free loop spaces}
\label{sec:Morse}

In this section we recall some results on Morse theory for the
energy integral on the Hilbert manifold model of the free loop space.
For details on this we refer to \cite{KlingSphere}.

Let $M$ be a compact Riemannian manifold equipped with the Levi-Civita
connection. We use the Hilbert manifold model of the free loop space $LM$. 
Write the circle as $S^1 = [0,1]/\{ 0, 1\}$. An element in $LM$ is 
an absolutely continuous map $f:S^1 \to M$ such that $f^\prime$ is square integrable
ie. $\int_0^1 |f^\prime (t)|^2 dt < \infty$. The Hilbert manifold model is 
homotopy equivalent to the usual continuous mapping space model.

The tangent space $T_f(LM)$ is the set of absolutely continuous 
tangent vector fields $X$ along $f$ such that the covariant derivative $DX(t)/dt$ 
is square integrable.  The free loop space $LM$ is equipped with a Riemannian metric 
$\inpl \cdot \cdot$ as follows:
\[
\inpl X Y = \int_0^1 \inp {\frac {DX} {dt}(t)} {\frac {DY} {dt}(t)} +
\inp {X(t)} {Y(t)} dt,
\]
where $X, Y \in T_f(LM)$.

The energy integral (or energy function) is defined by
\[ E:LM \to \RR; \quad 
E(f)=\frac 1 2 \int_0^1 |f^\prime(t)|^2 dt.  \]
The critical points for $E$ are precisely the closed geodesic on $M$. 
For a critical point $f$, the Hessian of $E$ has the following form:
$H_f(\cdot , \cdot) : T_f(LM)\times T_f(LM) \to \RR$;
\[
H_f(X,Y) = \int_0^1 \inp {\frac {DX} {dt}(t)} {\frac {DY} {dt}(t)} + 
\inp {R(X(t),f^\prime (t))f^\prime (t)} {Y(t)} dt, 
\]
where $R(\cdot , \cdot )\cdot$ denotes the curvature tensor on $M$.
The Hessian determines a self adjoint operator 
$A_f$ on $T_f (LM)$ satisfying $H_f (X, Y) = \inpl {A_f (X )} Y$ 
for all $X$ and $Y$. The operator $A_f$ is the sum of the identity 
with a compact operator, so there are at most a finite number of negative
eigenvalues, each corresponding to a finite dimensional vector space 
of eigenvectors of $A_f$. The kernel of $A_f$, which
is also finite dimensional, consists of the periodic Jacobi fields 
along $f$. 

Now let $N(e)$ be the space of critical points of $E$ with energy 
level $e$. It is known that $-\text{grad} \medspace E$ satisfy 
condition (C) of Palais and Smale so that one can do Morse 
theory on $LM$ if some additional non-degeneracy condition is 
satisfied. For us the so called Bott non-degeneracy condition is 
the relevant one. It requires firstly that for each critical 
value $e$ the space $N(e)$ is a compact submanifold 
of $LM$ and secondly that for each  $f \in N(e)$ the 
restriction of the Hessian $H_f$ to the complement 
$(T_f N(e))^\perp$ of $T_f N(e)$ in $T_f (LM)$ is non-degenerate.
The Bott non-degeneracy condition is a strong assumption
on the metric of $M$, but for instance the symmetric spaces
satisfy this, according to \cite[Theorem 2]{Ziller}.

Assume that the Bott non-degeneracy condition holds.
The negative bundle $\mu^- (e)$ over 
$N(e)$ is the vector bundle whose fiber at $f$ is the 
vector space spanned by the eigenvectors belonging to negative 
eigenvalues of $A_f$. Similarly, $\mu^0 (e)$ and 
$\mu^+ (e)$ are the vector bundles with fibers spanned by 
the eigenvectors corresponding to the eigenvalue $0$ and the 
positive eigenvalues respectively.

Let the critical values of the energy function be
$0=e_0 < e_1 <\dots $. Consider the filtration of
$LM$ given by $\FC (e_i) = E^{-1}([0 , e_i ])$.
This filtration is equivariant with respect to the action of
the circle. 

The tangent bundle of $LM$ restricted to
$N(e_i )$ splits ${\TT}$-equivariantly into a sum of three bundles.
\[
T(LM)|_{N(e_i )} \cong
\mu^- (e_i )\oplus \mu^0 (e_i )\oplus \mu^+ (e_i ).
\]
The standard Morse theory argument can be carried through
equivariantly on the Hilbert manifold $LM$. This was done 
by Klingenberg. For an account of this work see section 
\cite[2.4]{Klingenberg}, especially theorem 2.4.10. The 
statement of this theorem implies that we have an equivariant 
homotopy equivalence
\[
\FC (e_i )/\FC (e_{i-1} ) \simeq \Th (\mu^- (e_i )).
\]

\section{Klingenberg's calculation of negative bundles for projective spaces}
\label{sec:Klingenberg}

We will now focus on the projective spaces $P^n (\alpha )$ over the 
complex numbers $\CC$ for $\alpha = 2$, the quaternions $\HH$ for 
$\alpha = 4$ and the Cayley numbers $\OO$ for $\alpha = 8$. 
Note that $P^n(8)$ only exist when $n= 1$ or $n=2$.
These spaces are endowed with the Riemannian metric which makes them symmetric
spaces of rank one. This metric is determined up to a positive constant, which we 
fix by requiring the sectional curvature to have maximal value $2\pi^2$ and minimal 
value $\pi^2 /2$ \cite[1.1]{KlingProj}.

Klingenberg calculates the negative bundles for $L(P^n (\alpha ))$ in
\cite{KlingProj} and we will review this calculation.

Let $B_q(P^n(\alpha) )\subseteq LP^n(\alpha )$ denote the critical submanifold of 
$q$-fold covered primitive geodesics. A non-constant geodesics $f\in B_q(P^n(\alpha ))$ lies
on a unique projective line $S^\alpha \cong P^1(\alpha ) \subseteq P^n (\alpha )$.
For each $t\in [0,1]$ we split the tangent space at $f(t)$ into a horizontal 
subspace of tangent vectors to this projective line and its orthogonal 
complement called the vertical subspace \cite[1.3]{KlingProj}
\[ T_{f(t)}(P^n(\alpha )) = T_{f(t)}(P^n(\alpha ))_h \oplus T_{f(t)}(P^n(\alpha ))_v. \] 
The horizontal subspace has real dimension $\alpha$ and the vertical subspace 
has real dimension $\alpha (n-1)$. A tangent vector field 
$X \in T_f(P^n(\alpha ))$ decomposes into a horizontal component $X_h$ and a 
vertical component $X_v$ and this decomposition is compatible with the 
covariant derivative along $f$.

\begin{proposition}[Klingenberg]
Consider the parallel transport around a simple closed geodesics $f:[0,1] \to P^n(\alpha )$ 
with $f(0)=f(1)=p$. The horizontal subspace of $T_p(P^n (\alpha ))$ is carried into itself by the 
identity map. The vertical subspace is carried into itself by the reflection at the origin.
\end{proposition}
We will not review Klingenberg's proof here. A proof for the complex projective space will however 
appear later in Lemma \ref{lemma:vector fields}.

\begin{lemma}[Klingenberg]
\label{KlingenbergLemma}
Let $f\in B_q(P^n(\alpha ))$ where $q$ is a positive integer. 
The Hessian $H_f(\cdot , \cdot )$ on $T_f(LP^n(\alpha ))$ has eigenvectors 
as follows:
\begin{enumerate}
\item
\[ X_p(t) = A\cos (2\pi pt)+B\sin (2\pi pt), \quad p\in \NN_0 , \]
where $A$ and $B$ are constant (i.e. parallel) horizontal vector fields along $f$
such that $\inp A {f^\prime (t)} = \inp B {f^\prime (t)} = 0$ for all $t$. The
eigenvalue for $X_p$ is
\[ \lambda_p = \frac {4\pi^2 (p^2-q^2)} {1+4\pi^2 p^2}. \]
We write $E_{h,p}$ for the vector space formed by the $X_p$'s for a fixed $p$. 
It has real dimension $\alpha -1$ for $p=0$ and $2(\alpha -1)$ for $p>0$.
\item 
\[ Y_r(t) = A \cos (\pi rt) + B\sin (\pi rt), \quad r \in \NN_0, \quad
r\equiv q \text{ mod } 2, \]
where $A$ and $B$ are constant vertical vector fields along $f$. The 
eigenvalue of $Y_r$ is
\[ \mu_r = \frac {\pi^2 (r^2-q^2)} {1+\pi^2 r^2}. \]
We write $E_{v,r}$ for the vector space formed by $Y_r$. It has real dimension
$\alpha (n-1)$ if $r=0$ and $2\alpha (n -1)$ if $r>0$.
\item 
\[Z_s(t) = (a\cos (2\pi st)+b\sin (2\pi st))f^\prime (t), \quad s\in \NN_0 ,\]
where $a, b\in \RR$. The eigenvalue for $Z_s$ is
\[ \nu_s = \frac {4\pi^2 s^2} {1+4\pi^2 s^2}. \]
We write $E_{t,s}$ for the vector space formed by $Z_s$. It has real dimension
$1$ for $s=0$ and $2$ for $s>0$.
\end{enumerate}
\end{lemma}

\begin{proof}
The proposition above and the parity condition in (2) ensures that $X_p(0)=X_p(1)$ and $Y_r(0)=Y_r(1)$.

With our choice of metric, $|f^\prime (t)|^2 = 2q^2$. Moreover, the curvature 
tensor for $P^n(\alpha )$ is known, and its block matrix form allows 
Klingenberg to decompose the Hessian into a horizontal and a vertical quadratic 
form
\cite[1.4]{KlingProj}
\begin{align*}
H_f^h(X_h,Y_h) = & \int_0^1 \inp {\frac {DX_h} {dt} (t)} {\frac {DY_h} {dt} (t)} \\
& -2\pi^2 ( 2q^2 \inp {X_h(t)} {Y_h(t)} - 
\inp {f^\prime (t)} {X_h(t)} \inp {f^\prime (t)} {Y_h (t)}) dt, \\
H_f^v(X_v,Y_v) = & \int_0^1 \inp {\frac {DX_v} {dt} (t)} {\frac {DY_v} {dt} (t)}
- \pi^2 q^2 \inp {X_v(t)} {Y_v(t)} dt.
\end{align*}

Consider the eigen equation $H_f^h(X_h,Y_h)=\lambda \inpl {X_h} {Y_h}$ for all
$Y_h$ where $\lambda \in \RR$. If $X_h$ possess second covariant derivative, we get an 
equivalent equation via partial integration
\begin{equation} 
\label{h-eigen}
(1-\lambda)\frac {D^2X_h} {dt^2} + 
(4\pi^2q^2+\lambda)X_h -2\pi^2\inp {f^\prime} {X_h} f^\prime = 0.
\end{equation}
We insert $X_p$ in this equation. 
Since $\frac {D^2 X_p} {dt^2} = -4\pi^2 p^2 X_p$ we get the following:
\[((4\pi^2 p^2 + 1)\lambda - 4\pi^2 (p^2-q^2))X_p = 0.\]
Thus $\lambda_p$ is an eigenvalue for $H_f^h(\cdot , \cdot)$ with eigenvector $X_p$.

From $H_f^v(X_v,Y_v)=\mu \inpl {X_v} {Y_v}$ for all $Y_v$ where $\mu \in \RR$,
we get the eigen equation
\begin{equation}
\label{v-eigen}
(1-\mu )\frac {D^2X_v} {dt^2} + (\pi q^2+\mu )X_v = 0.
\end{equation}
We insert $Y_r$. Since $\frac {D^2 Y_r} {dt^2} = -\pi^2 r^2 Y_r$ we get
\[ ((\pi^2 r^2+1)\mu -\pi^2 (r^2 - q^2))Y_r = 0. \]
Thus $\mu_r$ is an eigenvalue for $H_f^v(\cdot , \cdot)$ with eigenvector $Y_r$.

Finally, we insert $Z_s$ into (\ref{h-eigen}). Since $f$ is a geodesics we have
that $\frac {Df} {dt} = 0$. Thus, $\frac {D^2 Z_s} {dt^2} = -4\pi^2 s^2 Z_s$ and
we obtain
\[ ((1+4\pi^2 s^2)\lambda - 4\pi^2 s^2)Z_s = 0.\]
We see that $\nu_s$ is an eigenvalue for $H_f^h(\cdot , \cdot)$ with 
eigenvector $Z_s$. 
\end{proof}

The subspaces described in 1.-3. have trivial pairwise intersection.
They also generate the full Hilbert space $T_f(P^n(\alpha ))$, so we have
the following result:

\begin{corollary}
The negative subspace is the direct sum
\[ T_f(LP^n(\alpha ))^- = \bigoplus_{0\leq p<q} E_{h,p} \oplus
\bigoplus_{0\leq r<q, \medspace r\equiv q \text{ mod } 2} E_{v,r}.\]
It has real dimension $(2q-1)(\alpha -1) + (q-1)\alpha (n-1)$.
The zero subspace is
\[ T_f(LP^n(\alpha ))^0 = E_{t,0} \oplus E_{h,q} \oplus E_{v,q}. \]
It has real dimension $2\alpha n -1$. The positive subspace is the Hilbert 
direct sum
\[ T_f(LP^n(\alpha ))^+ = \bigoplus_{p>q} E_{h,p} \oplus
\bigoplus_{r>q, \medspace r\equiv q \text{ mod } 2} E_{v,r} \oplus \bigoplus_{s>0} E_{t,s}.\]
\end{corollary}

Klingenberg shows that there are vector bundles over 
$B_q(P^n(\alpha ))$ for $q\geq 1$ as follows: 

\begin{center}
\begin{tabular}{|l|c|l|l|}
\hline
Vector bundle & $\text{dim}_{\RR}$ & Fiber over $f$ & Condition \\
\hline
$\eta_{h,0}$ & $\alpha -1$ & $E_{h,0}$ & \\
$\sigma_{h,p}$ & $2(\alpha -1)$ & $E_{h,p}$ & $p\geq 1$\\
$\sigma_{v,2p-1}$ & $2\alpha (n-1)$ & $E_{v,2p-1}$ & $q$ odd, $p\geq 1$\\
$\eta_{v,0}$ & $\alpha (n-1)$ & $E_{v,0}$ & $q$ even\\
$\sigma_{v,2p}$ & $2\alpha (n-1)$ & $E_{v,2p}$ & $q$ even\\
\hline
\end{tabular}
\end{center}

Thus, we have the following result \cite[1.6]{KlingProj}:
\begin{theorem}[Klingenberg]
\label{Klingenberg}
The non-trivial critical points for the energy integral 
$E:L(P^n(\alpha ))\to \RR$ decompose into the non-degenerate critical 
submanifolds $B_q(\alpha ) = B_q(P^n(\alpha ))$ consisting of the $q$-fold
covered parametrized great circles, $q=1,2,\dots$ ; $E(B_q(\alpha ))=2q^2$. 
The negative bundle $\mu_q^-$ over $B_q(\alpha )$ has the following form:
\begin{align*}
& \mu_q^- = \eta_{h,0} \oplus \bigoplus_{p=1}^{q-1} \sigma_{h,p} \oplus 
\bigoplus_{p=1}^{\frac {q-1} 2} \sigma_{v,2p-1} &\text{for $q$ odd,} \\
& \mu_q^- = \eta_{h,0} \oplus \bigoplus_{p=1}^{q-1} \sigma_{h,p} \oplus 
\eta_{v,0} \oplus \bigoplus_{p=1}^{\frac {q-2} 2} \sigma_{v,2p} 
&\text{for $q$ even.}
\end{align*}
\end{theorem}

\section{Spaces of geodesics viewed as projective Stiefel manifolds}
\label{sec:geod}

From now on, we consider the complex projective space $\CP^n$. It has a
Hermitian metric, which we now describe. References are \cite{KN2} page 273 
or \cite{MT} page 142.

Equip $\CC^{n+1}$ with the standard Hermitian inner product
$h(v,w)=\sum_{k=1}^{n+1}v_k\overline w_k$. The real part 
$g^\prime (v,w) = \Re h(v,w)$
is the usual inner product on $\RR^{2n+2}\cong \CC^{n+1}$. Furthermore, 
$h(v,w)=g^\prime (v,w)+ig^\prime (v,iw)$.

Let $S^{2n+1}=\{ x\in \CC^{n+1} | h(x,x)=1 \}$ be the unit sphere and write
$\TT$ for the unit circle group. Consider the Hopf projection
\[ \hp :S^{2n+1} \to S^{2n+1}/\TT =\CP^n . \]
By restriction of $h$ we have a Hermitian inner product on the orthogonal complement 
$( \CC x )^\perp = \{ v\in \CC^{n+1} | h(x,v)=0 \}$ and
$( \CC x )^\perp$ is a real subspace of the tangent space 
$T_x(S^{2n+1})$. One can equip $\CP^n$ with a Hermitian metric 
$\tilde h (\cdot , \cdot )$ such that
\[ \fs_x : (\CC x)^\perp \subseteq T_x(S^{2n+1}) \xrightarrow{\hp_*} 
T_{\hp (x)}(\CP^n ) \]
becomes a $\CC$-linear isometry. The following identity holds
\begin{equation} 
\label{eqn:circleact}
\fs_{zx}(zv)=\fs_x(v) \text{ for } z\in \TT .
\end{equation}
The real part $\tilde g (\cdot , \cdot ) = \Re \tilde h (\cdot , \cdot )$ 
is the Fubini-Study metric on $\CP^n$. 
(In \cite{KN2} they allow a rescaling of $\tilde g$ by $4/c$ for a 
positive constant $c$. We let $c=4$.) It is known that the sectional 
curvature for this metric has maximal value $4$ and minimal value $1$ when $n>1$. 
Thus the metric on $\CP^n$ used in section \ref{sec:Klingenberg} is 
$\frac {\pi^2} 2 \tilde g$.

For $\CP^n$ with Riemannian metric $\tilde g$ and associated Levi-Civita connection, 
we now describe the spaces of closed geodesics $B_q(\CP^n )$ in terms of projective Stiefel manifolds.
Recall that $B_q(\CP^n )$ is the space of constant geodesics
for $q=0$, primitive geodesics for $q=1$ and $q$-fold iterated primitive
geodesics for $q\geq 2$.

\begin{definition}
Let $\St 2 {n+1}$ denote the Stiefel manifold of complex orthonormal 2-frames in $\CC^{n+1}$. 
\end{definition}

Write $U$ for the unitary matrix
\[ U = \frac 1 {\sqrt 2} \begin{bmatrix} 1 & -i \\ 1 & i \end{bmatrix}  \]
and let $D_\theta$ and $R_\theta$ be the following diagonal and rotation matrices:
\[ 
D_\theta = \begin{bmatrix} e^{-i\theta} & 0 \\ 0 & e^{i\theta} \end{bmatrix}, \quad 
R_\theta = \begin{bmatrix} \cos \theta & -\sin \theta \\ \sin \theta & \cos \theta \end{bmatrix} . 
\] 

\begin{lemma}
\label{lemma:diagonalization}
Matrix multiplication defines a right action
\begin{align*}
& \St 2 {n+1} \times U(2) \to \St 2 {n+1};  \\
& \big( (u,v), \begin{bmatrix} a & b \\ c & d\end{bmatrix} \big) \mapsto (au+cv, bu+dv).
\end{align*}
The diffeomorphism $\tau : \St 2 {n+1} \to \St 2 {n+1}$; $(u,v)\mapsto (u,v)U$ satisfies
\[ \tau ((u,v)D_\theta ) = \tau (u,v) R_\theta . \]
\end{lemma}

\begin{proof}
Regarding the action, it suffices to verify that the image frame is orthonormal. By the elementary properties of the 
inner product, one finds that
\[h(au+cv,au+cv)=1, \quad h(au+cv, bu+dv)=0, \quad h(bu+dv, bu+dv)=1 .\]
so this is the case. Let
\[ V=U^{-1}=\frac 1 {\sqrt 2} \begin{bmatrix} 1 & 1\\ i & -i \end{bmatrix} .\]
One has
\[ 
\begin{bmatrix} \alpha & -\beta \\ \beta & \alpha \end{bmatrix}  \begin{bmatrix}1 \\ i \end{bmatrix} =
(\alpha -i\beta ) \begin{bmatrix} 1 \\i \end{bmatrix} , \quad
\begin{bmatrix} \alpha & -\beta \\ \beta & \alpha \end{bmatrix}  \begin{bmatrix}1 \\ -i \end{bmatrix} =
(\alpha +i\beta ) \begin{bmatrix} 1 \\ -i \end{bmatrix} .
\]
For $\alpha = \cos \theta$ and $\beta = \sin \theta$ this gives us the diagonalization $V^{-1} R_\theta V = D_\theta$.
Thus, $UR_\theta = D_\theta U$ such that $\tau$ has the stated property.
\end{proof}

We now define a right action of the torus group $\TT^2$ on the Stiefel manifold. We use different notations for the
left and right circle group factors as follows: $\TT^2 = \TT \times U(1)$. We view $\TT$ and $U(1)$ as subgroups of the abelian group $\TT^2$ via inclusion in the first and second factor respectively. 
For each integer $q$ there is a group homomorphism
\[ \iota_q : \TT^2 \to U(2); \quad (z_1, z_2) \mapsto 
\begin{bmatrix} z_1^q z_2 & 0 \\ 0 & z_2 \end{bmatrix}. \]
Recall that a right $G$-space $X$ is considered a left $G$-space by the action $g*x = x*g^{-1}$ for 
$g\in G$, $x\in X$ and vice versa.

\begin{definition} 
The torus $\TT^2$ acts from the right on $\St 2 {n+1}$ via the homomorphism $\iota_q$ 
and the $U(2)$-action of Lemma \ref{lemma:diagonalization}. Let $\Sta 2 {n+1} q$ denote the corresponding left
$\TT^2$-space.  
The projective Stiefel manifold is defined as the quotient space
\[ \PSt 2 {n+1} q = \Sta 2 {n+1} q / U(1). \] 
It is equipped with a left action 
of the quotient group $\TT \cong \TT^2/U(1)$.
When viewed as a space without a group action, the projective Stiefel manifold is denoted $\PStnull 2 {n+1}$. 
\end{definition}

\begin{remark}
Alternatively, we have
\[ \PStnull 2 {n+1} = \St 2 {n+1} / diag_2 (U(1)) \]
where $diag_2 (U(1)) \subseteq U(2)$ denotes the diagonal inclusion. The  $\TT$-action is given by 
\[  z*[u,v] = [z^{-q}u, v] = [u, z^qv]=[c^{-q}u, c^{q}v] \]
where $c$ is a square root of $z$.
Note that $[u,zv]=[u,v] \Rightarrow z=1$ so the $\TT$-action is free when $q=1$.
\end{remark}

The projective Stiefel manifold is diffeomorphic to the sphere bundle of the tangent bundle 
of $\CP^n$ as follows: 
\[
\xymatrix@C=1 cm{
\Phi : \PStnull 2 {n+1} \ar[r]^-{\cong } & S(T(\CP^n )); & [u,v] \ar@{|->}[r] & (\fs_u (v))_{\hp (u)} .
}  
\]
So via the exponential map it corresponds to a space of geodesics. The $\TT$-action on $\PSt 2 {n+1} 1$ 
corresponds to complex rotation in the tangent bundle since $\Phi ([u,zv]) =  (z\fs_u (v))_{\hp (u)}$. 
The purpose of the diffeomorphism $\tau$ of Lemma \ref{lemma:diagonalization} 
is to make this $\TT$-action, which has a simple description, correspond to rotation of closed geodesics. 
More precisely we have:

\begin{theorem}
\label{thm:diffeo}
For every positive integer $q$ there is a $\TT$-equivariant diffeomorphism
\[ \phi_q : \PSt 2 {n+1} q \to \B q n ; \quad
\phi_q ([u,v])(t) = 
\hp ( \frac {e^{-q\pi i t}u + e^{q\pi i t}v} {\sqrt 2} ). \]
\end{theorem}

\begin{proof}
It is well known (\cite{GHL} 2.110 or \cite{KN2} page 277) that there is a diffeomorphism
\begin{align*}
& \psi_q : \PStnull 2 {n+1} \to \B q n ; \\
& \psi_q ([a,b])(t) = \hp \big( \cos (q\pi t)a+\sin (q\pi t)b \big) = 
\hp \big( (a,b)R_{q\pi t} \begin{bmatrix} 1 \\ 0 \end{bmatrix} \big) , 
\end{align*}
where $0\leq t \leq 1$. The diffeomorphism becomes equivariant when we let $\TT$ act on
$\B q n$ and $\PStnull 2 {n+1}$ by
$(e^{2\pi is}*f)(t)= f(s+t)$ and $e^{2\pi is}\star [a,b] = [(a,b)R_{q\pi s}]$ respectively. 
Write $\PSt 2 {n+1} {(q)}$ for the projective Stiefel manifold equipped with this action.
 
The group $diag_2(U(1))$ is in the center of $U(2)$ so the map $\tau$ from Lemma \ref{lemma:diagonalization} 
gives us a well-defined automorphism of the projective Stiefel manifold. This automorphism is a $\TT$-equivariant map
\[ \tau_q : \PSt 2 {n+1} q \to \PSt 2 {n+1} {(q)} \]
by the equation for $\tau$ proven in Lemma \ref{lemma:diagonalization}. 
Via Euler's formulas we find
\[ (\psi_q \circ \tau_q)([u,v])(t) = 
\rho (\frac {e^{-iq\pi t}u+e^{iq\pi t}v} {\sqrt 2}). \]
Thus, $\psi_q \circ \tau_q = \phi_q$ and we have the desired result.
\end{proof}

\section{A description of the negative bundle}
\label{sec:Negative_bundle}

In this section we will describe the negative bundles as bundles over 
projective Stiefel manifolds. We start by the following result regarding
the constant (parallel) horizontal and vertical vector fields mentioned
in Lemma \ref{KlingenbergLemma}.

\begin{lemma}
\label{lemma:vector fields}
Let $(u,v)\in \St 2 {n+1}$ and let $q$ be a positive integer. 
Define the curve 
\[ c:[0,1]\to S^{2n+1}; 
\quad c(t) = \frac {e^{-q\pi i t}u+e^{q\pi i t}v} {\sqrt 2} \] 
and put $f(t)=\hp (c(t))=\phi_q([u,v])(t)$. Then the horizontal and vertical subspaces
at $f(t)$ are given by
\[ T_{f(t)}(\CP^n )_h = \fs_{c(t)}(\linspan_\CC (c^\prime (t))), \quad
T_{f(t)}(\CP^n )_v = \fs_{c(t)}(\{ u, v \}^\perp ), \]
where $\perp$ is with respect to the Hermitian inner product $h$.
Furthermore,
\[ H(t) = \fs_{c(t)}(e^{-q\pi i t}u-e^{q\pi i t}v) \]
is a parallel and horizontal vector field along $f$, such that
$\tilde g (H(t),f^\prime (t))=0$ for all $t$, and
\[ V(w)(t) = \fs_{c(t)}(w) \]
is a parallel and vertical vector field along $f$ for 
all $w\in \{ u, v \}^\perp$. These vector fields satisfy 
\[ H(0) = H(1), \quad V(w)(0) = (-1)^q V(w)(1).\]
\end{lemma}

\begin{proof}
We have that $c^\prime (t)= -q\pi i(e^{-q\pi i t}u-e^{q\pi i t}v)/\sqrt 2$. Since $u$ and $v$ are 
orthonormal vectors it follows that $h(c^\prime (t), c^\prime (t))=q^2\pi^2$ and $h(c(t), c^\prime (t))=0$. 
Furthermore, $\{ c(t), c^\prime (t) \}^\perp = \{ u, v\}^\perp$ for all $t$. 
Thus we have an orthogonal decomposition
\[ \{ c(t) \}^\perp = \linspan_\CC (c^\prime (t))\oplus 
\{ c(t), c^\prime (t) \}^\perp
= \linspan_\CC (c^\prime (t))\oplus 
\{ u, v \}^\perp . \]
By the chain rule 
$f^\prime (t) = T_{c(t)}(\hp )(c^\prime (t)) = \fs_{c(t)}(c^\prime (t))$
such that
\[ T_{f(t)}(\CP^n )_h =\linspan_\CC (f^\prime (t)) = \eta_{c(t)} (\linspan_\CC ( c^\prime (t))) \]
and since $\fs_{c(t)}$ is an isometry, we also obtain the desired descriptions of 
the vertical subspace.

Put $\tilde H(t) = e^{-q\pi i t}u-e^{q\pi i t}v$.
Since $\tilde H$ is a complex rescaling of $c^\prime$ we see that $H$ is 
a horizontal vector field. 

We have equipped $S^{2n+1}\subseteq \CC^{n+1} \cong \RR^{2n+2}$
with the Riemannian metric induced from $\RR^{2n+2}$. Since $c$ is a geodesics in that metric
we have $\frac {D} {dt} \tilde H(t)=0$.
The projective space $\CP^n$ is equipped with the Fubini-Study metric 
so it follows that $\frac {D} {dt} H(t)=0$.
Thus $H$ is a parallel vector field along $f$. 

We have 
$h(\tilde H(t), c^\prime (t))= 
-q\pi i \vert \vert \tilde H(t) \vert \vert^2 /\sqrt 2$. The real part of this equation
gives us that $g^\prime (\tilde H(t), c^\prime (t))=0$.
It follows that $\tilde g( H(t), f^\prime (t))=0$ since 
$\fs_{c(t)}$ is an isometry.

By the first part of the lemma, $V(w)$ is a vertical vector field
for all $w\in \{ u, v \}^\perp$. Since $w$ is constant, $\frac {dw} {dt}=0$. So its 
orthogonal projection $\frac {Dw} {dt}$ onto the tangent space at $c(t)$ is also zero. It follows that
$\frac {DV(w)} {dt}=0$ such that $V(w)$ is a parallel vector field along $f$. The final relations follows by 
equation (\ref{eqn:circleact}).
\end{proof}

We will now give a slightly different description of the curve and vector fields of the lemma such that the 
proof of Theorem \ref{neg-isos} becomes easier.

\begin{definition}
\label{definition:c(u,v)}
For $(u,v)\in \St 2 {n+1}$ we define the {\em closed} geodesic
\[ c(u,v): \TT \to S^{2n+1}; \quad 
c(u,v)(z) = \frac 1 {\sqrt 2} (z^{-1}u+zv). \]
\end{definition}

The equivariant diffeomorphism 
$\phi_q: \PSt 2 {n+1} q \to \B q n$ from Theorem \ref{thm:diffeo} is 
defined by the diagram
\[
\xymatrix@C=1.5 cm{
\TT \ar[r]^{c(u,v)} & S^{2n+1} \ar[d]^{\hp} \\
\TT \ar[u]^{(\sqrt \cdot )^q } \ar[r]^{\phi_q ([u,v]) } & \CP^n. } 
\]
Note that $h(c(u,v),c(u,-v))=0$. So we can view $c(u,-v)$ as a vector field along $c(u,v)$.

\begin{definition}
Define a parallel horizontal tangent vector field along $\phi_2 ([u,v])$ by
\[ H(u,v)(z) = \fs_{c(u,v)(z)}(c(u,-v)(z)) \]
and for $w \in \{ u,v \}^\perp$, where $\perp$ is with respect to $h$, a 
parallel vertical tangent vectors field by 
\[ V(u,v,w)(z) = \fs_{c(u,v)(z)}(w). \]
\end{definition}

The relations to the curve and vector fields of Lemma \ref{lemma:vector fields} are as follows:
\[ c(u,v)(e^{q\pi i t}) = c(t), \quad H(u,v)(e^{q\pi i t}) = H(t), \quad V(u,v,w)(e^{q\pi i t})=V(w)(t) .\]

\begin{proposition} \label{HVrel}
For all $\lambda \in U(1)$ one has the identities
\[ H(\lambda u, \lambda v)=H(u,v), \quad V(\lambda u, \lambda v, \lambda w)=V(u,v,w). \]
Furthermore, for all $z_1, z_2 \in \TT$ one has
\[ H(u,v)(z_1z_2) = H(u, z_1^2v)(z_2), \quad V(u,v,w)(z_1z_2)=V(u, z_1^2v, z_1w)(z_2). \]
As special cases, $H(u,v)(-z)=H(u,v)(z)$ and $V(u,v,w)(-z)=V(u,v,-w)(z)$.
\end{proposition}

\begin{proof}
The first two identities follows by equation (\ref{eqn:circleact}). From Defintion \ref{definition:c(u,v)} one sees that 
\begin{equation*}
c(u,v)(z_1z_2)=c(z_1^{-1}u,z_1v)(z_2) \label{c(u,v) 2}
\end{equation*}
This relation and the first two identities gives the last two identities.
\end{proof}

We now have sufficient information on the constant horizontal and vertical vector fields in
Klingenberg's Lemma \ref{KlingenbergLemma}. Next we will define the bundles over 
projective Stiefel manifolds which correspond to the summands of the negative bundle.

The concept of $G$-vector bundles (over the real or complex numbers), for a topological group $G$, 
will be used (\cite{Atiyah} \S 1.6).
A $G$-space $E$ is a $G$-vector bundle over a $G$-space $X$ if 
\begin{enumerate}
\item[(i)] $E$ is a vector bundle over $X$,
\item[(ii)] The projection $E\to X$ is a $G$-map,
\item[(iii)] For each $g\in G$ the map $g\cdot : E_x \to E_{gx}$ is a vector space homomorphism.
\end{enumerate}
In the special case where the action of $G$ on $X$ is trivial, we see that each fiber becomes a $G$-module.

\begin{proposition}
\label{proposition:G-vb}
Let $G$ be a compact Lie group with a closed normal subgroup $H\subseteq G$. Let $X$ be a $G$-space 
such that the canonical projection $X\to X/H$ is a principal $H$-bundle. 
\begin{enumerate}
\item If $\eta \to X$ is a $G$-vector bundle then $\eta /H \to X/H$ is a $G/H$-vector bundle.
\item For $G$-vector bundles $\eta_1 \to X$ and $\eta_2 \to X$ there is a natural isomorphism of
$G/H$-vector bundles
\[ (\eta_1 \oplus \eta_2 )/H \cong \eta_1 /H \oplus \eta_2 /H. \]
\item If $\xi_1 \to Y$ and $\xi_2 \to Y$ are $G$-vector bundles and $f:X\to Y$ is a $G$-map then
there is a natural isomorphism of $G/H$-vector bundles
\[ f^* (\xi_1 \oplus \xi_2) /H \cong f^*(\xi_1 )/H \oplus f^*(\xi_2 )/H . \]
\end{enumerate}
\end{proposition} 

\begin{proof}
(1) Let $p:E\to X$ be the projection map for $\eta$. By \cite{tomD1} I.3.4 there is a $G/H$-action
on $E/H$ such that the following diagram commutes:
\[
\xymatrix@C=1 cm{
G\times E \ar[r] \ar[d] & E \ar[d] \\
G/H \times E/H \ar[r] & E/H. } 
\]
Likewise we have a $G/H$-action on $X/H$ and $p/H$ is a $G/H$-map by naturality. 
Thus, condition (ii) holds.

Furthermore, $p/H : E/H \to X/H$ is a vector bundle by \cite{tomD1} I.9.4, such that (i) holds, and there is a 
pullback diagram of vector bundles
\[
\xymatrix@C=1 cm{
E \ar[r] \ar[d]^-{p} & E/H \ar[d]^-{p/H} \\
X \ar[r] & X/H. } 
\]
Finally, the first of the diagrams above gives us a commutative diagram of fibers for $x\in X$ and $g\in G$:
\[
\xymatrix@C=1 cm{
E_x \ar[r]^-{g\cdot } \ar[d]^-{\cong } & E_{gx} \ar[d]^-{\cong} \\
(E/H)_{[x]} \ar[r]^-{[g]\cdot} & (E/H)_{[gx]}. } 
\]
The top map is linear since $E\to X$ is a $G$-vector bundle. The vertical maps are isomorphisms by the pullback
diagram above. So condition (iii) also holds.

(2) There is a well-defined map $\psi$ which makes the following diagram commute:
\[
\xymatrix@C=1 cm{
\eta_1 \oplus \eta_2 \ar[r] \ar@{=}[d] & (\eta_1 \oplus \eta_2 )/H \ar[d]^-{\psi} \\
\eta_1 \oplus \eta_2 \ar[r] & \eta_1 /H \oplus \eta_2 /H . } 
\]
The bottom map is surjective so $\psi$ is also surjective. Furthermore, $\psi$ is a bundle map over $X/H$ 
which maps a fiber of its domain to an isomorphic fiber of its codomain by the pullback diagram above. 
So $\psi$ is an isomorphism of vector bundles. One sees directly by its transformation rule
$\psi ([v_1 , v_2]) = ([v_1], [v_2])$ that $\psi$ is a $G/H$-map. 

(3)  The standard isomorphism $f^*(\xi_1 \oplus \xi_2) \cong f^* (\xi_1) \oplus f^* (\xi_2)$ is $G$-equivariant.
so we have an isomorphism $f^*(\xi_1 \oplus \xi_2)/H \cong  (f^* (\xi_1) \oplus f^* (\xi_2))/H$ of $G/H$-vector
bundles. The result then follows by (2).
\end{proof}

The projection map $\St 2 {n+1} \to \St 2 {n+1} / diag_2 (U(1)) = \PStnull 2 {n+1}$ is a principal $U(1)$-bundle
by standard arguments. So by (1) in the propositon above, we have the following construction of $\TT$-vector
bundles:

\begin{definition}
Let $f:\Sta 2 {n+1} q \to X$ be a $\TT^2$-map and let $\xi$ be a complex $\TT^2$-vector bundle 
over $X$. Form the pullback $f^*(\xi )$. The quotient $f^* (\xi )/ U(1)$ is a complex $\TT$-vector bundle 
which we denote  
\[ \PV_{2,q} (f,\xi ) \to \PSt 2 {n+1} q. \]
\end{definition}

We only need this construction for a special type of torus vector bundles.
\begin{definition}
Let $\eta \to X$ be a complex vector bundle and $i$, $j$ two integers. Equip the total space of $\eta$ by a 
$\TT^2$-action via complex multiplication in the fibers as follows:
\[ (z_1, z_2)*v=z_1^iz_2^jv .\]
The resulting $\TT^2$-vector bundle over the trivial $\TT^2$-space $X$ is denoted $\eta (i,j) \to X$.
\end{definition}

Let $\gamma_2$ be the canonical bundle over the Grassmannian 
$\Gr 2 {n+1}$. Its total space consists of the pairs $(V, v)$ where $V$ is a complex two dimensional subspace
of $\CC^{n+1}$ and $v \in V$. It has an orthogonal complement bundle $\gamma_2^\perp$ over $\Gr 2 {n+1}$ 
consisting of pairs $(V, w)$ where $w \in V^\perp \subseteq \CC^{n+1}$. Let
$\pi : \Sta 2 {n+1} q \to \Gr 2 {n+1}$ be the projection which maps 
a frame to its complex span. We equip the Grassmannian with the trivial $\TT^2$-action such that
$\pi$ becomes equivariant. Finally, for a complex vector space $V$, we wite $\overline V$ for its conjugate 
vectorspace. As real vector spaces $V$ and $\overline V$ are the same but $z\cdot v=\overline zv$ for
$v\in V$ and $z\in \CC$. For a complex vector bundle $\xi$ we write $\overline \xi$ for its conjugate
vector bundle.

\begin{definition} \label{def:nu}
For $r=q$ mod $2$ we define $\TT$-vector bundles as follows:
\[
\nu_{r,q} = \PV_{2,q} (\pi, \gamma_2^\perp (\frac {r+q} 2, 1) ) , \quad
\overline{\nu}_{r,q} = \PV_{2,q} (\pi , \overline { \gamma_2^\perp } (\frac {r-q} 2, -1 )) .
\]
\end{definition}

Two product bundles also enter in the description. For a 
$\TT$ representation $V$ we let $\epsilon_q (V)$ denote the 
product bundle $pr_1:\PSt 2 {n+1} q \times V \to \PSt 2 {n+1} q$. 
Let $\CC {(s)}$ for $s\in \ZZ$ denote the complex numbers $\CC$ equipped with the
$\TT$-action $z*\lambda = z^s\lambda$, and equip the real numbers $\RR$ with 
the trivial $\TT$-action. The product bundles which enter are
$\epsilon_q (\RR )$ and $\epsilon_q (\CC {(p)} )$. Note that 
$\epsilon_q (\RR )$ is a real $\TT$ vector bundle and that the others 
are complex $\TT$ vector bundles.

Write $\real (z)$ for the real part of a complex number $z$.
We have the following result, where the summands in Klingenbergs Theorem \ref{Klingenberg} have been labeled by 
an additional index $q$ indicating that they are vector bundles over $B_q(\CP^n)$.
\begin{theorem}
\label{neg-isos} 
Let $p$, $q$ and $r$ be positive integers with $p<q$ and $r<q$.
There are isomorphisms of $\TT$-vector bundles over the $\TT$-equivariant diffeomorphism 
\[ \phi_q : \PSt 2 {n+1} q \to \B q n\]
as follows, where $h_q$ is defined for $q=0$ mod $2$ and $k_{r,q}$ is defined for $r=q$ mod $2$: 
\begin{align*}
& f_q: \epsilon_q (\RR ) \to \eta_{h,0,q} \, ; 
& & f_q ([u,v],s)(z) = s H(u,v)((\sqrt z)^q), \\ 
& g_q: \epsilon_q (\CC {(p)} ) \to \sigma_{h,p,q} \, ;
& & g_q ([u,v],\lambda )(z) = 
\real (\lambda z^{p}) H(u,v)((\sqrt z)^q), \\
& h_q: \nu_{0,q} \to \eta_{v,0,q} \, ;
& & h_q ([u,v,w])(z) = V(u,v,w)((\sqrt z)^q), \\
& k_{r,q}: \nu_{r,s} \oplus {\overline \nu}_{r,s} \to \sigma_{v,r,q} \, ;  
& & k_{r,q} ([u,v,w_1,w_2])(z) = \\
& & & V(u,v,(\sqrt z)^rw_1+(\sqrt z)^{-r}w_2)((\sqrt z)^q).   
\end{align*}
In the last formula, $\sqrt z$ appears twice. One must use the same choice of square root in 
both places.
\end{theorem}

\begin{proof}
For all four maps, the real dimension of the fiber of the domain equals the
real dimension of the fiber of the codomain. So it suffices to show that
each map is well-defined, surjective on fibers and $\TT$-equivariant.

The map $f_q$ is independent of the choice of representative 
for the class $[u,v]$ and the choice of square root of $z$ by Proposition \ref{HVrel}. So it is well-defined. 
By Lemma \ref{KlingenbergLemma} and Lemma \ref{lemma:vector fields}, $f_q$ 
is surjective on fibers. 
By proposition \ref{HVrel} we see that it is $\TT$-equivariant as follows:
\begin{align*}
f_q([u,v],s)(z_1z_2) &= s H(u,v)((\sqrt{z_1})^q (\sqrt{z_2} )^q) \\
&= s H(u, z_1^qv)((\sqrt{z_2} )^q) 
= f_q(z_1*[u,v],s)(z_2).  
\end{align*}

The map $g_q$ is well-defined by Proposition \ref{HVrel}.
For complex numbers $z_1 = \alpha_1+i\beta_1$ 
and $z_2 = \alpha_2+i\beta_2$ written in standard form
we have $\real (z_1 \overline z_2) = \alpha_1 \alpha_2 + \beta_1 \beta_2$. So for
$\lambda = \alpha + i \beta$ and $z=e^{-2\pi i t}$ we get
\[ \real (\lambda  z^{p}) = \alpha \cos (2\pi p t) + \beta \sin (2\pi p t) \]
such that $g_q$ is surjective on fibers by Lemma \ref{KlingenbergLemma} and
Lemma \ref{lemma:vector fields}. 
We see that $g_q$ is $\TT$-equivariant as follows: 
\begin{align*}
g_q([u,v],\lambda )(z_1z_2) &= 
\real (\lambda (z_1z_2)^{p}) f_q([u,v],1)(z_1z_2)\\
&= \real (z_1^p \lambda z_2^p) f_q(z_1*[u,v],1)(z_2)
= g_q(z_1*([u,v], \lambda ))(z_2).  
\end{align*}

The map $h_q$ is well-defined for $q$ even by Proposition \ref{HVrel}. It is surjective on fibers by
Lemma \ref{KlingenbergLemma} and Lemma \ref{lemma:vector fields}. By Proposition \ref{HVrel} we see 
that $h_q$ is $\TT$-equivariant as follows: 
\begin{align*}
h_q([u,v,w])(z_1z_2) &= V(u,v,w)((\sqrt{z_1} )^q (\sqrt{z_2} )^q ) 
= V(u, z_1^qv, z_1^{\frac q 2}w )((\sqrt{z_2} )^q) \\
& = h_q(z_1*[u,v,w])(z_2)  
\end{align*}

Finally, consider the map $k_{r,q}$ where $r=q$ mod $2$. For $\lambda \in U(1)$ we have
\[ [u,v,w_1,w_2]=[\lambda u, \lambda v, \lambda w_1, \lambda^{-1} \cdot w_2] =
[\lambda u, \lambda v, \lambda w_1, \lambda w_2]. \]
So by Proposition \ref{HVrel} the map $k_{r,q}$ is independent on the choice of representative 
of the class $[u,v,w_1,w_2]$. By the last remark of the proposition it is also independent of the choice
of square root of $z$ and hence it is well-defined. 

For $z=e^{2\pi i t}$ with choice of square root $\sqrt z = e^{\pi i t}$ we have
\begin{align*} 
& k_{r,q}([u,v,w_1,w_2])(z) = V(u,v,e^{\pi r it}w_1+e^{-\pi rit}w_2)(e^{\pi qit}) \\
& = \cos (r\pi t) V(u,v,w_1+w_2)(e^{\pi qit}) + \sin (r\pi t) V(u,v, i(w_1-w_2))(e^{\pi qit}).
\end{align*}
For a given pair of vectors $a$ and $b$ in $\{ u, v \}^\perp$ the two equations
$w_1+w_2=a$ and $i(w_1-w_2)=b$ have the solution
$w_1 = \frac 1 2 (a-ib)$, $w_2 = \frac 1 2 (a+ib)$.
Comparing with Lemma \ref{KlingenbergLemma} and Lemma
\ref{lemma:vector fields} we see that the surjectivity on fibers holds.

Finally, we check that $k_{r,q}$ is $\TT$-equivariant. Firstly, by Proposition \ref{HVrel} we have
\begin{align*}
& k_{r,q}([u,v,w_1,w_2])(z_1z_2) \\ 
&= V(u,v,(\sqrt {z_1})^{r}(\sqrt{z_2} )^{r}w_1+(\sqrt {z_1})^{-r}(\sqrt{z_2} )^{-r}w_2)((\sqrt{z_1})^q (\sqrt{z_2} )^q)\\
&=  V(u, z_1^qv, (\sqrt {z_1})^{r+q}(\sqrt{z_2} )^{r}w_1+(\sqrt {z_1})^{-(r-q)}(\sqrt{z_2} )^{-r}w_2)((\sqrt{z_2} )^q).
\end{align*}
Secondly,
\[ z_1*[u,v,w_1,w_2]=[u, z_1^qv, z_1^{\frac {r+q} 2}w_1, z_1^{\frac {r-q} 2}\cdot w_2]
= [u, z_1^qv, (\sqrt{z_1})^{r+q}w_1, (\sqrt{z_1})^{-(r-q)}w_2] \]
such that $k_{r,q}(z_1*[u,v,w_1,w_2])(z_2)$ equals the above expression.
\end{proof}

Combining Theorem \ref{Klingenberg} and Theorem \ref{neg-isos} we obtain our first main result:

\begin{theorem}
\label{main1}
For every positive integer $q$, there are isomorphisms of $\TT$-vector bundles as follows:
\begin{align*}
& \mu_q^- \cong \epsilon_q (\RR ) \oplus \bigoplus_{0<s<q} \epsilon_q (\CC (s)) \oplus 
\bigoplus_{\substack{0<r<q \\ r = q \text{ mod } 2}} (\nu_{r,q} \oplus \overline{\nu}_{r,q})  &\text{for $q$ odd,} \\
& \mu_q^- \cong \epsilon_q (\RR ) \oplus \bigoplus_{0<s<q} \epsilon_q (\CC (s)) \oplus \nu_{0,q} \oplus 
\bigoplus_{\substack{0<r<q \\ r = q \text{ mod } 2}} (\nu_{r,q} \oplus \overline{\nu}_{r,q})
&\text{for $q$ even.}
\end{align*}
\end{theorem}

\section{Projective bundles and Borel constructions}
\label{sec:PV_2 bundles}

In this section we establish results which are aimed at calculating 
characteristic classes of the Borel construction with respect to the 
$\TT$-action of the negative bundle.

\begin{proposition}
\label{PV properties}
The following statements hold:
\begin{enumerate}
\item Let $\xi_1 \to X$ and $\xi_2 \to X$ be $\TT^2$-vector bundles and let
$f: \Sta 2 {n+1} q \to X$ be a $\TT^2$-map. Then there is an isomorphism of $\TT$-vector bundles
\begin{equation*}
\PV_{2,q} (f, \xi_1 \oplus \xi_2 ) \cong 
\PV_{2,q} (f, \xi_1 ) \oplus \PV_{2,q} (f, \xi_2 ).
\end{equation*}
\item Write $\epsilon^k_Y$ for the trivial $k$-dimensional complex vector bundle over a space $Y$. Equip $Y$ 
with the trivial $\TT^2$-action. Let $g: \Sta 2 {n+1} q \to Y$ be a $\TT^2$-map   
and let $t: \Sta 2 {n+1} q \to *$ denote the map to a point. Then for all integers $i$ and $j$ one has 
\begin{equation*}
\PV_{2,q} (g, \epsilon^k_Y (i,j)) = \PV_{2,q} (t, \epsilon^k_* (i,j)) =: \PV_{2,q} (\epsilon^k (i,j)).
\end{equation*}
Furthermore, there is a decomposition 
\begin{equation*}
\PV_{2,q} (\epsilon^k (i,j)) = \bigoplus_{m=1}^k\PV_{2,q} (\epsilon^1 (i,j)).
\end{equation*}
\end{enumerate} 
\end{proposition}

%

\begin{proof}
(1) This is a special case of Proposition \ref{proposition:G-vb} (3).

(2) Both pullbacks $f^*(\epsilon^k_Y (i,j))$ and $t^*(\epsilon^k_* (i,j))$ gives the same $\TT^2$-vector
bundle. Its projection map is $pr_1: \Sta 2 {n+1} q \times \CC^k \to \Sta 2 {n+1} q$ and the action on 
the total space is given by $(z_1, z_2)*((u,v),w)=((z_2u,z_1^q z_2v), z_1^iz_2^jw)$. This observation gives us
the first part of the statement. The last part follows from (1).
\end{proof}

Let $\gamma_1 \to \CP^n$ be the canonical line bundle. Its total space consists of pairs $(V, v)$ where
$V$ is a complex one dimensional subspace of $\CC^{n+1}$ and $v\in V$. The projection map sends
$(V,v)$ to $V\in \CP^n$. Sometimes we use the quotient space model $S^{n+1}/U(1)$ for $\CP^n$ instead.
Then a point in $\CP^n$ is written as $[u]$ where $u\in S^{2n+1}$.

\begin{definition} 
Let $\pi_i$ for $i=1,2$ be the composite maps
\[
\xymatrix@C=1.5 cm{
\pi_i :\Sta 2 {n+1} q \ar[r] & \PSt 2 {n+1} q \ar[r]^-{pr_i} & \CP^n,
}\] 
where $pr_1 ([u,v])=[u]$ and $pr_2([u,v])=[v]$. Note that $\pi_i$ is a $\TT^2$-map where
the action on $\CP^n$ is trivial. For $r=q$ mod $2$ we
define one dimensional complex $\TT$-vector bundles over $\PSt 2 {n+1} q$ by
\begin{align*}
L_{0,r,q} &= \PV_{2,q} (\epsilon^1 (\frac {r+q} 2 , 1)), 
& {\overline L}_{0,r,q} &= \PV_{2,q} ({\overline \epsilon}^1 (\frac {r-q} 2 , -1)), \\ 
L_{1,r,q} &= \PV_{2,q} (\pi_1 , \gamma_1 (\frac {r+q} 2 , 1)), 
& {\overline L}_{1,r,q} &= \PV_{2,q} (\pi_1 , {\overline \gamma_1} (\frac {r-q} 2 , -1)), \\ 
L_{2,r,q} &= \PV_{2,q} (\pi_2 , \gamma_1 (\frac {r+q} 2 , 1)), 
& {\overline L}_{2,r,q} &= \PV_{2,q} (\pi_2 , {\overline \gamma_1} (\frac {r-q} 2 , -1)).
\end{align*}
\end{definition}

\begin{theorem} \label{sumiso}
There are $\TT$-equivariant isomorphisms for
$r=q$ mod $2$ as follows:
\begin{align*}
& \nu_{r,q} \oplus L_{1,r,q} \oplus L_{2,r,q} \cong 
L_{0,r,q}^{\oplus (n+1)} , \\
& \overline{\nu}_{r,q} \oplus \overline{L}_{1,r,q} \oplus \overline{L}_{2,r,q} \cong 
\overline{L}_{0,r,q}^{\oplus (n+1)} .
\end{align*}
\end{theorem}

\begin{proof}
We give the proof of the first isomorphism. The proof of the second is similar. Put $m=\frac {r+q} 2$.
By Propositions \ref{PV properties} we have
\begin{align*}
\PV_{2,q} (\pi , \gamma_2^\perp (m,1))\oplus \PV_{2,q} (\pi , \gamma_2 (m,1)) & \cong
\PV_2(\pi , \gamma_2^\perp (m,1) \oplus \gamma_2 (m,1)) \\
& \cong \PV_2 (\pi , \epsilon^{n+1} (m,1) ) \cong L_{0,r,q}^{\oplus (n+1)}. 
\end{align*}
Thus it suffices to show that $\PV_{2,q} (\pi , \gamma_2 (m,1))\cong L_{1,r,q} \oplus L_{2,r,q}$.
The standard isomorphism
\[
\pi_1^* (\gamma_1) \oplus \pi_2^* (\gamma_1) \to \pi^* (\gamma_2); \quad 
((u,v,w_1), (u,v,w_2)) \to (u,v,w_1+w_2),
\]
where $w_1\in \linspan_\CC (u)$ and $w_2\in \linspan_\CC (v)$, is $\TT^2$-equivariant with respect to
our actions. The result follows by Proposition \ref{proposition:G-vb} (2).
\end{proof}

We will now give pullback descriptions of the $\TT$-line bundles. 
The following notation is used: For a complex vector bundle 
$\xi \to X$ and integer $m\in \ZZ$ we put 
$\xi (m) = \xi$ where $z\in \TT \subseteq \CC$ acts on each fiber
by multiplication with $z^m$. Thus, $\xi (m) \to X$ is a $\TT$-vector
bundle over a trivial $\TT$-space.

\begin{proposition} \label{lbpb}
Let $\epsilon^1 \to \CP^n$ be the trivial line bundle and $\gamma_1\to \CP^n$ the canonical line bundle. 
There are pullback diagrams of $\TT$-vector bundles as follows for $r=q$ mod $2$ and $i=1,2$:
\[
\xymatrix@C=1.0 cm{
L_{i,r,q} \ar[r] \ar[d]
& \epsilon^1 (\frac {r+(-1)^{i+1}q} 2)  \ar[d] 
& L_{0,r,q} \ar[r] \ar[d] 
& \overline \gamma_1 {(\frac {r+(-1)^{i+1}q} 2)} \ar[d] \\
\PSt 2 {n+1} q \ar[r]^-{pr_i} & \CP^n
& \PSt 2 {n+1} q \ar[r]^-{pr_i} & \CP^n 
}\]
\[
\xymatrix@C=1.0 cm{
\overline{L}_{i,r,q} \ar[r] \ar[d]
& \overline{\epsilon}^1 (\frac {r+(-1)^{i}q} 2)  \ar[d] 
& \overline{L}_{0,r,q} \ar[r] \ar[d] 
& \gamma_1 {(\frac {r+(-1)^{i}q} 2)} \ar[d] \\
\PSt 2 {n+1} q \ar[r]^-{pr_i} & \CP^n
& \PSt 2 {n+1} q \ar[r]^-{pr_i} & \CP^n 
}\]
\end{proposition}

It might seems strange that eg. the bundle $L_{1,r,q}$ is a trivial $\TT$-vector 
bundle over $\PSt 2 {n+1} q $ as stated. It did not come from a trivial bundle but from $\gamma_1$.
The 'untwisting' appears when the quotient is formed in the construction of $L_{1,r,q}$ as a result of
the definition of the $U(1)$-action.

\begin{proof}
Regarding the upper left pullback diagram for $i=1$, 
the bundle map over $pr_1$ is defined by
\[ f_1: L_{1,r,q} \to \CP^n \times \CC ;\quad [u,v,w]\mapsto ([u],k(w,u)) , \]
where $k(w,u)\in \CC$ is the scalar determined by $w=k(w,u)u$. The following properties hold
for $w_1, w_2 \in \linspan_{\CC}(u)$, $a_1, a_2 \in \CC$ and $b\in U(1)$:
\begin{align*} 
& k(a_1w_1+a_2w_2,u)=a_1k(w_1,u)+a_2k(w_2,u) ,\\
& k(w,bu) =b^{-1} k(w,u) .
\end{align*}
It follows that $k(zw,zu)=k(w,u)$ for $z\in U(1)$ so the bundle map $f_1$ is well-defined:
\[ [zu,zv,zw]\mapsto ([zu],k(zw,zu))=([u],k(w,u)). \]
Furthermore, $f_1$ is a fiber-wise $\CC$-linear isomorphism, so we have a pullback. We check that
$f_1$ is $\TT$-equivariant as well: Put $m=\frac {r+q} 2$. Then,
\[ f_1(z*[u,v,w])=f_1([u,z^qv,z^mw])=([u], k(z^mw,u))
= ([u],z^m k(w,u)). \]
Similarly, the bundle map 
$f_2: L_{2,r,q} \to \CP^n \times \CC$; $[u,v,w]\mapsto ([v],k(w,v))$,
where $w=k(w,v)v$, gives us the upper left pullback diagram for $i=2$. In this case, the $\TT$-equivariance follows 
from the computation 
\[ f_2(z*[u,v,w])=f_2([u,z^qv,z^mw])=([z^qv], k(z^mw, z^qv))
= ([v],z^{m-q} k(w,u)). \]

The bundle maps in the lower left diagram are still $f_1$ and $f_2$, but with conjugate complex structure
on domain and target. For $i=1$, we have
\begin{align*} 
& f_1(z*[u,v,w])=f_1([u,z^qv,z^{m-q}\cdot w])=f_1([u,z^qv,z^{-m+q} w]) \\
& = ([u], k(z^{-m+q}w,u))
= ([u], z^{-m+q}k(w,u))=([u], z^{m-q}\cdot k(w,u)). 
\end{align*}
Thus, $f_1$ is $\TT$-equivariant. A similar argument gives us that $f_2$ is $\TT$-equivariant so we 
have the stated pullback diagrams for $i=1,2$.

The bundle map in the upper right diagram for $i=1$ is defined by
\[ g_1: L_0 \to \overline \gamma_1 ; \quad
[u,v,k] \mapsto ([u], k\cdot u)=([u], \overline k u). \]
It is well-defined because $z\overline z = 1$ for $z\in U(1)$ such that
\[ [zu,zv,zk] \mapsto ([zu], \overline z \overline k zu) = 
([u], \overline k u). \]
Since $g_1$ is a fiber-wise isomorphism, we have a pullback. $g_1$ is
also $\TT$-equivariant:
\[ g_1(z*[u,v,k])=g_1([u,z^qv,z^mk])=([u],z^{-m}\overline k u) = ([u], z^m \cdot \overline k u). \]
Similarly, the bundle map
$g_2: L_0 \to \overline \gamma_1$; 
$[u,v,k] \mapsto ([v], \overline k v)$
gives us the upper right pullback diagram for $i=2$.

The bundle maps $g_1$ and $g_2$ with conjugate complex structure on domain and target, gives the lower right
pullback diagrams.
\end{proof}

We are interested in the vector bundle $E\TT \times_\TT \mu_q^-$. Fortunately, forming Borel constructions of 
$G$-vector bundles is well behaved with respect to Whitney sums and pullbacks. One has the following standard results:
\begin{proposition} \label{Borel_vb_prop}
\label{orbit}
Let $G$ be a compact Lie group and let $\xi$, $\eta$ be 
$G$-vector bundles over a $G$-space $X$. Then there is a natural 
isomorphism
\[
\xymatrix@C=1.5 cm{
EG\times_G (\xi \oplus \eta ) \ar[r]^-{\cong} 
& (EG\times_G \xi )\oplus (EG\times_G \eta ).
}\] 
Furthermore, if $f:Y\to X$ is a $G$-map, then there is a natural 
isomorphism
\[
\xymatrix@C=1.5 cm{
EG\times_G f^* (\xi ) \ar[r]^-{\cong} 
& (EG\times_G f)^* (EG\times_G \xi ).
}\] 
\end{proposition}

%

\begin{corollary}
\label{fiberpullback}
Let $G$ be a compact Lie group and $p: \xi \to X$ a $G$-vector bundle over a trivial $G$-space $X$. Write 
$EG\to BG$ for the universal principal $G$-bundle, and let $i_1:BG\to BG\times X$ be the inclusion
$b\mapsto (b,x_0)$ where $x_0\in X$. Then there is a pullback diagram
\[
\xymatrix@C=1.5 cm{
EG\times_G p^{-1}(x_0) \ar[d]^{pr_1} \ar[r]  & EG\times_G \xi \ar[d]^{EG\times_G p} \\
BG \ar[r]^{i_1} & BG\times X 
}\] 
\end{corollary} 


\section{Characteristic classes}
\label{sec:characteristic}

In this section we compute the Chern classes of the vector bundles $(\mu_q^-)_{h\TT}$.
By Theorem \ref{sumiso} and Proposition \ref{lbpb} the following result is relevant:

 \begin{proposition} \label{BasicChern}
Let $x=c_1 (\gamma_1 )$ and $u=c_1 (\gamma_1^\infty )$ be the first Chern classes of the canonical line bundles
$\gamma_1 \to \CP^n$ and $\gamma_1^\infty \to \CP^\infty = B\TT$ such that
\[ H^*(B\TT \times \CP^n ; \ZZ  ) = \ZZ [u] \otimes \ZZ [x]/(x^{n+1}) .\]
Let $\epsilon^1 \to \CP^n$ be the trivial line bundle. Then for every $m\in \ZZ$ we have
\begin{align*}
c_1 (E\TT \times_\TT \gamma_1 {(m)} ) &= mu\otimes 1+1\otimes x, \\
c_1 (E\TT \times_\TT \epsilon^1 {(m)}) ) &= mu\otimes 1.
\end{align*}
\end{proposition}

\begin{proof}
We start by proving the following claim:
\[c_1 (E\TT \times_{\TT} \CC (m)) = mu. \]
The first Chern class defines a group homomorphism
\[ c_1: (\text{Vect}^1_\CC (B\TT ) ,\otimes , \overline{( \, )} )  \to (H^2(B\TT ; \ZZ ),+, -)\] 
which is in fact an isomorphism since $B\TT$ is homotopy equivalent to the CW-complex $\CP^\infty$ 
(see \cite[page 250]{Husemoller} or \cite{Hatcher}). There are isomorphisms of 
vector bundles for every $n$ as follows:
 \begin{align*}
& S^{2n-1}\times_{\TT} \CC (1) \to \gamma_1 ; & [v,z]&\mapsto (\linspan_\CC (v), zv), \\
& S^{2n-1}\times_{\TT} \CC (-1) \to \overline \gamma_1 ; &  [v,z]&\mapsto (\linspan_\CC (v), \overline z v). 
\end{align*}
Thus, we have isomorphisms $E\TT \times_\TT \CC (1) \cong \gamma_1$ and 
$E\TT \times_\TT \CC (-1) \cong \overline \gamma_1$.
Note that $\CC (0)$ equals $\CC$ with trivial $\TT$-action and for $k>0$, we have that
$\CC (k) \cong \otimes_{i=1}^k \CC (1)$ and 
$\CC (-k)\cong \otimes_{i=1}^k \CC (-1)$. 
We get corresponding tensor product decompositions of the vector bundles 
$E\TT \times_\TT \CC (m)$. The claim follows.
 
Choose base points in $B\TT$ and $\CP^n$, and consider the associated inclusions
\[i_1 : B\TT \to B\TT \times \CP^n , \quad i_2 : \CP^n \to B\TT \times \CP^n .\]
By Corollary \ref{fiberpullback} the pullback of both $\gamma_1 (m)_{h\TT }$ and 
$\epsilon^1 (m)_{h\TT }$ along $i_1$ equals the line bundle $E\TT \times_\TT \CC (m) $. Thus, 
\[ i_1^*(c_1({\gamma_1 (m)}_{h\TT }))=i_1^*(c_1({(\epsilon^1 (m))}_{h\TT }))=
c_1(E\TT \times_\TT \CC (m)) = mu. \]
The pullback of ${\gamma_1 (m)}_{h\TT }$ along 
$i_2: \CP^n \to E\TT \times \CP^n \to  B\TT \times \CP^n$ equals $\gamma_1$ and the pullback 
of $\epsilon^1(m)_{h\TT }$ along $i_2$ is the trivial line bundle $\epsilon^1$. Thus, 
\[ i_2^*(c_1({\gamma_1 (m)}_{h\TT }))=x, \quad i_2^*(c_1(\epsilon^1(m)_{h\TT }))=0. \]

Finally,  $H^2 (B\TT  \times \CP^n ; \ZZ )$ is generated by the two classes $u\otimes 1$, $1\otimes x$ and 
\begin{align*}
 i_1^*(u\otimes 1) = i_1^* \circ pr_1^* (u) = u & , &  i_2^* (u\otimes 1) = i_2^* \circ pr_1^* (u)=0 , \\
i_1^*(1\otimes x)=i_1^* \circ pr_2^* (x) = 0 & , &  i_2^* (1\otimes x) =i_2^* \circ pr_2^* (x)=x,
\end{align*}
so we have the desired result. 
\end{proof}

\begin{remark} \label{BasicChern1}
For any complex vector bundle $\xi$ one has that
\[ \overline{E\TT \times_\TT \xi (m)} =  E\TT \times_\TT \overline \xi (-m), \]
since in both cases, we mod out by the equivalence relation $(ez,v)\sim (e,z^mv)$, and we have 
the conjugate complex structure. So by the above result
\begin{align*}
c_1 (E\TT \times_\TT \overline \gamma_1 {(m)} ) &= mu\otimes 1-1\otimes x, \\
c_1 (E\TT \times_\TT \overline \epsilon^1 {(m)}) ) &= mu\otimes 1.
\end{align*}
\end{remark}

In order to use the pullback diagrams of Proposition \ref{lbpb}, we must compute the induced maps in cohomology of 
the two projection maps
\[ (pr_i)_{h\TT} : {\PSt 2 {n+1} q}_{h\TT} \to ( \CP^n )_{h\TT} = B\TT \times \CP^n, \quad i=1,2. \]
The mod $p$ cohomology of the domain space was computed in \cite{BO5}. We will need some of the 
results, leading to this calculation.

Let $\pi : \PP (\gamma_2 ) \to \Gr 2 {n+1}$ denote the projective bundle of the canonical bundle 
$\gamma_2 \to \Gr 2 {n+1}$. 
We can describe the total space as a set of flags:
\[ \PP (\gamma_2 ) = \{ V_1 \subseteq V_2 \subseteq \CC^{n+1} | \dim_\CC (V_i) = i \} .\]
By \cite{BO5} Lemma 2.6, we have an isomorphism
\[
\xymatrix@C=1.0 cm{
\psi: \PSt 2 {n+1} 1 /\TT \ar[r]^-{\cong}  &  \PP (\gamma_2 ); \quad [u,v]\TT \mapsto (\linspan_\CC (u) \subseteq \linspan_\CC (u,v) \subseteq \CC^{n+1}) . 
}\] 
There is a canonical line bundle $\lambda \to \PP (\gamma_2 )$ with complement line bundle
$\lambda^\prime \to \PP (\gamma_2 )$ as follows:
\[ \lambda = \{ (V_1\subseteq V_2, v) | v\in V_1 \} , \quad
\lambda^\prime = \{ (V_1\subseteq V_2, w) | w\in V_1^\perp \subseteq V_2 \} .\]
There are pullback diagrams
\[
\xymatrix@C=1.5 cm{
\lambda \ar[r] \ar[d] & \gamma_1  \ar[d] & \lambda^\prime \ar[r] \ar[d] & \gamma_1 \ar[d] \\
\PP (\gamma_2 ) \ar[r]^-{p_1} & \CP^n & \PP (\gamma_2) \ar[r]^-{p_2} & \CP^n , 
}\]
where $p_1(V_1 \subseteq V_2 ) = V_1$ and $p_2(V_1\subseteq V_2)=V_1^\perp$.
Note also that $\lambda \oplus \lambda^\prime \cong \pi^* (\gamma_2 )$. We have the following slightly enhanced
version of Theorem 3.2 in \cite{BO5}:

\begin{theorem} \label{Q_n}
There is an isomorphism of graded rings
\[ H^* (\PP (\gamma_2 );\ZZ ) \cong \ZZ [x_1, x_2]/(Q_n, Q_{n+1}), \]
where $x_1$ and $x_2$ have degree $2$ and for positive integers $k$, 
\[ Q_k (x_1, x_2) = \sum_{i=1}^k x_1^ix_2^{k-i} = \frac {x_1^{k+1}-x_2^{k+1}} {x_1-x_2} . \]
Furthermore, $p_1^* (x) = x_1$ and $p_2^* (x) = x_2$.
\end{theorem}

\begin{proof}
The ring structure is given in Theorem 3.2 of \cite{BO5}. From the proof of this theorem one has that
\[ x_1=c_1(\lambda ),\quad x_2=\pi^* (c_1(\gamma_2 ))-c_1(\lambda ), \quad 
\pi^* (c_1(\gamma_2 ))=x_1+x_2. \]
Thus, $p_1^*(x)=p_1^*(c_1(\gamma_1 ))=c_1(p_1^*(\gamma_1 ))= c_1(\lambda ) = x_1$ and
\[x_1+x_2=c_1(\pi^* (\gamma_2 ))=c_1(\lambda \oplus \lambda^\prime )=c_1(\lambda )+c_1(\lambda^\prime )=
x_1+c_1(\lambda^\prime )\]
such that $x_2=c_1(\lambda^\prime )$.
\end{proof}

Recall that a left $G$-space $X$ is also a right $G$ space with action $x*g=g^{-1}*x$ for $x \in X$, $g\in G$.
For the right $\TT$-space $\PSt 2 {n+1} 1$ we have the following result:
\begin{lemma} \label{Euler}
The principal $\TT$-bundle $\rho : \PSt 2 {n+1} 1 \to \PSt 2 {n+1} 1 /\TT$ has associated complex line bundle
$\lambda \otimes_\CC \overline{\lambda^\prime}$. That is, we have an isomorphism of line bundles
\[ \xymatrix@C=1.5 cm{
\PSt 2 {n+1} 1 \times_\TT \CC (1) \ar[d] \ar[r]^-{\cong} & \lambda \otimes_\CC \overline{\lambda^\prime} \ar[d] \\
\PSt 2 {n+1} 1 /\TT \ar[r]^-{\psi}_-{\cong} & \PP (\gamma_2 )
} \]
The Euler class of $\rho$ is
\[ e(\rho ) = x_1 - x_2. \]
\end{lemma}

\begin{proof}
The bundle map over the isomorphism $\psi$ is defined by
\[ [[u,v],k] \mapsto ((\linspan_\CC (u)\subseteq \linspan_\CC (u,v)), k(u\otimes v)). \]
We check that this is a well-defined map. Firstly, the linear span is unchanged by a rescaling of the 
generators by nonzero scalars. Secondly, for $z\in U(1)$ we have $[u,v]=[zu,zv]$ in the projective
Stiefel manifold, but also
\[ zu\otimes zv = zu\otimes \overline z\cdot v = z\overline z(u \otimes v) = u\otimes v. \]
Thirdly, for $z\in \TT$ we have $[[u,v],zk]=[[u,v]*z,k]=[[zu,v],k]$ but also
$zk(u\otimes v) = k(zu\otimes v)$ which completes the argument. The bundle map is an isomorphism on fibers. 

The Euler class of $\rho$ equals the first Chern class of 
the associated line bundle, which is 
$c_1 (\lambda \otimes_\CC \overline{\lambda^\prime}) = c_1(\lambda ) -c_1(\lambda^\prime )=x_1-x_2$.
\end{proof}

\begin{remark}
By the lemma above we get a sphere bundle interpretation of the projective Stiefel manifold
\[ \PSt 2 {n+1} 1 = \PSt 2 {n+1} 1 \times_\TT \TT = S(  \PSt 2 {n+1} 1 \times_\TT \CC (1)) \cong 
S(\lambda \otimes_\CC \overline{\lambda^\prime} ). \]
Thus, there is an isomorphism of left $\TT$-spaces for every $q\in \ZZ$:
\[ \PSt 2 {n+1} q \cong S\big( (\lambda \otimes_\CC \overline{\lambda^\prime} )(-q)\big) . \]
\end{remark}

For a left $\TT$-space $X$ with action map $\mu : \TT \times X \to X$, we can twist the action by the
power map $\lambda_n : \TT \to \TT ; \lambda_n (z)=z^n$ and obtain another $\TT$-space $X^{(n)}$.
Thus the underlying spaces of $X$ and $X^{(n)}$ are equal, but the action map for $X^{(n)}$ is
$\mu_n: \TT \times X^{(n)}\to X^{(n)}$; $\mu_n (z,x)=\mu (\lambda_n (x),z)$.

\begin{proposition} \label{twist}
Let $X$ be a left $\TT$-space and let $C_n$ denote the cyclic group of order $n$. 
There is a vertical and horizontal pullback of fibration sequences which is natural in $X$ as follows:
\[
\xymatrix@C=1.5 cm{
\star \ar[r] \ar[d] & X \ar@{=}[r] \ar[d] & X \ar[d] \\
BC_n \ar[r] \ar@{=}[d] & E\TT \times_\TT X^{(n)} \ar[r]^-{E(\lambda_n) \times_\TT id} \ar[d]^-{pr_1} 
& E\TT \times_\TT X \ar[d]^-{pr_1} \\
BC_n \ar[r] & B\TT \ar[r]^-{B(\lambda_n)} & B\TT  
}\]
Assume furthermore that the right $\TT$-space associated to $X$ gives a principal $\TT$-bundle $\rho: X \to X/\TT$. 
Write it as a pullback of the universal bundle
$E\TT \to B\TT$ along a map $f: X/\TT \to B\TT$. Then the right vertical projection map in the diagram
above can be replaced by $f$ in the following sense: There is a diagram, which commutes up to homotopy, and 
where $pr_2$ is a weak homotopy equivalence
\[ \xymatrix@C=1.5 cm{
E\TT \times_\TT X \ar[r]^-{pr_2}_-{\simeq} \ar[d]^-{pr_1} & X/\TT \ar[dl]^-{f} \\
B\TT
}\]
Finally, if we let $e(\rho )$ denote the Euler class, the two maps 
\[ \xymatrix@C=1.5 cm{
H^*(B\TT ;\ZZ ) \ar[r]^-{pr_1^*} & H^*(E\TT \times_\TT X^{(n)} ; \ZZ ) & H^*(X/\TT ; \ZZ ) \ar[l]_-{pr_2^*}
}\]
satisfy that 
\[ pr_1^* (nu) = pr_2^*( e(\rho )). \]
\end{proposition}

\begin{proof}
A proof for the first pullback diagram can be found in \cite{BO5} Lemma 6.1.
Regarding the second diagram, first note that $pr_2$ is a fibration with contractible fiber $E\TT$ and hence
a weak homotopy equivalence. In order to verify that the diagram commutes up to homotopy, it
suffices to check, that the right triangle in the following diagram commutes up to homotopy:
\[ \xymatrix@C=1.5 cm{
E\TT \times_\TT X \ar[r]^-{E\TT \times_\TT \hat f} \ar[d]^-{pr_2} & E\TT \times_\TT E\TT \ar[r]^-{pr_1} \ar[d]^-{pr_2} 
& B\TT \ar@{=}[dl]  \\
X/\TT \ar[r]^f & B\TT
}\]
Both $pr_1$ and $pr_2$ in the triangle are homotopy equivalences. By the diagrams
\[ \xymatrix@C=1.0 cm{
E\TT \times_\TT E\TT \ar[r]^-{pr_1} \ar[d]^-{tw} & B\TT \ar@{=}[d]  
& [E\TT \times_\TT E\TT, B\TT ] \ar[r]^-{\cong} \ar[d]^-{tw^*} & H^2 (E\TT \times_\TT E\TT ;\ZZ ) = \ZZ \ar[d]^-{tw^*} \\
E\TT \times_\TT E\TT \ar[r]^-{pr_1} & B\TT 
& [E\TT \times_\TT E\TT, B\TT ] \ar[r]^-{\cong} & H^2 (E\TT \times_\TT E\TT ;\ZZ ) = \ZZ 
}\]
it suffices to see that $tw^* = id : \ZZ \to \ZZ$. The twist gives a self map of the fibration
\[ \TT \to E\TT \times E\TT \to E\TT \times_\TT E\TT . \]
It is the identity on the fiber of the point $[a,a]$ since the twist also changes the sides of the actions on both factors.
By the long exact sequence of homotopy groups, one sees that $tw_*  =id$ on
$\pi_2 (E\TT \times_\TT E\TT )$. By Hurewicz and universal coefficients, the result follows for cohomology.

We have that $f^*(u)=e(\rho )$. In the second diagram of the theorem, this gives us that
$pr_1^*(u)=pr_2^*(e(\rho))$. Combining this with the first diagram, the last statement follows.
\end{proof}

\begin{proposition} \label{pr_i}
There is a commutative diagram for $i=1,2$ where $\pi_1$ and $\pi_2$ denotes projection on first and 
second factor:
\[
\xymatrix@C=1.5 cm{
\PP (\gamma_2 ) \ar[r]^-{p_i} & \CP^n \\
\PSt 2 {n+1} 1 /\TT \ar[r]^-{pr_i/\TT} \ar[u]^-{\cong} & \CP^n \ar@{=}[u]\\
E\TT \times_\TT \PSt 2 {n+1} q \ar[u]^-{\pi_2}  \ar[r]^-{E\TT \times_\TT pr_i} \ar[d]_-{\pi_1} 
& B\TT \times \CP^n \ar[u]^-{\pi_2} \ar[d]_-{\pi_1} \\
B\TT \ar@{=}[r] & B\TT
}\] 
In cohomology with $\ZZ$-coefficients, one has that
\[ (E\TT \times_\TT pr_i)^* (1\otimes x)  = \pi_2^* (x_i) \text{ and }
(E\TT \times_\TT pr_i)^* (qu\otimes 1) = \pi_2^* (x_1-x_2). \]
\end{proposition}

\begin{proof}
Only the top square in the diagram requires an argument and it commutes by direct verification.
The first equation follows by the diagram.
The second follows by Lemma \ref{Euler} and Proposition \ref{twist}.
\end{proof}

We can now prove the following enhanced version of \cite{BO5} Theorem 4.1:
\begin{theorem} \label{pr_i1}
Let $n$ and $q$ be integers with $n>1$ and $q>0$. Let $p$ be a prime. There is an isomorphism 
\[
H^*_\TT (\PSt 2 {n+1} q ;\FF_p )\cong 
\begin{cases}
\FF_p [x_1, x_2]/(Q_n,Q_{n+1}) , & p\nmid q, \\
\FF_p [u,x,\sigma ]/(x^{n+1}, \sigma^2 ), & p\mid q, \, p\mid (n+1),\\
\FF_p [u,x,\overline \sigma ]/(x^{n}, \overline \sigma^2 ), & p\mid q, \, p\nmid (n+1),
\end{cases}
\]
where the classes $u$, $x$, $x_1$, $x_2$ have degree $2$ and $\deg (\sigma ) = 2n-1$, $deg(\overline \sigma )=2n+1$.
The polynomials $Q_k\in \FF_p [x_1, x_2]$ are defined as follows for positive integers $k$:
\[ Q_k (x_1, x_2) = \sum_{i=0}^k x_1^i x_2^{k-i} . \]
The maps
\[ pr_i^* : H^* (B\TT \times \CP^n ; \FF_p ) \to H^*_\TT (\PSt 2 {n+1} q ; \FF_p ) \]
are given by the following for $i=1, 2$:
\begin{align*}
u\otimes 1 &\mapsto \frac 1 q (x_1-x_2), & 1\otimes x &\mapsto x_i, & \text{for } p & \nmid q, \\
u\otimes 1 &\mapsto u, & 1\otimes x &\mapsto x, & \text {for } p &\mid q.
\end{align*}
\end{theorem}

\begin{proof}
The computation of the cohomology ring is given in \cite{BO5} Theorem 4.1. We review parts of the proof in order
to include the description of the projection maps.

By proposition \ref{twist}, we have a pullback of fibration sequences
\[ \xymatrix@C=1.5 cm{
BC_q \ar[r] \ar@{=}[d] & E\TT \times_\TT \PSt 2 {n+1} q \ar[r]^-{\pi_2} \ar[d]^-{\pi_1} 
& \PSt 2 {n+1} 1 /\TT  \ar[d]^-{f} \\
BC_q \ar[r] & B\TT \ar[r]^-{B(\lambda_q)} & B\TT  
}\] 
Consider the associated Serre spectral sequences. We have trivial coefficients in both of these since the base of 
the lower fibration is simply connected. 

Assume that $p\nmid q$. Then, $H^*(BC_q; \FF_p)=\FF_p$, 
and by the upper spectral sequence $\pi_2$ induces an isomorphism in cohomology. 
The results follows by Theorem \ref{Q_n} and Proposition \ref{pr_i} via universal coefficients.

Assume that $ p\mid q$. One has that $H^* ( BC_q; \FF_p )= \FF_p [v,w]/I_{p,q}$, where the degrees are
$|v|=1$, $|w|=2$
and $I_{p,q}$ is the ideal $(v^2-w)$ for $p=2$, $4\nmid q$ and the ideal $(v^2)$ otherwise.
The $E_2$-page of the Serre spectral sequence for the upper fibration has the form
\[ E_2^{**} = \FF_p [x_1,x_2]/(Q_n, Q_{n+1}) \otimes  \FF_p [v,w]/I_{p,q}, \]
where the bi-degrees are $||x_1||=||x_2||=(2,0)$, $||v||=(0,1)$, $||w||=(0,2)$. Via the spectral sequence
for the lower fibration sequence, one finds that $d_2(w)=0$, $d_2(v)=x_1-x_2$ and that $w$ is a 
permanent cycle. It follows that $E_3=E_\infty$.

We let $K_n$ and $C_n$ denote the kernel and cokernel of multiplication with $(x_1-x_2)$ on
$\FF_p [x_1,x_2]/(Q_n, Q_{n+1})$. Then
\[ E_\infty^{**}=E_3^{**} = (C_n\oplus vK_n)\otimes \FF_p [w].\]
In \cite{BO5}, proof of Theorem 4.1, the kernel and cokernel are analyzed further, and one obtains
the following bigraded algebra description of the $E_\infty$-page:
\[ E_\infty^{**} = \begin{cases}
\FF_p [w,x_1,\sigma ]/(x_1^{n+1}, \sigma^2 ), & p\mid (n+1), \\
 \FF_p [w,x_1,\overline \sigma ]/(x_1^{n}, \overline \sigma^2 ), & p\nmid (n+1). 
\end{cases} \]
Here $x_1$ denotes the class $[x_1]$ which equals the class $[x_2]$ since $d_2(v) =x_1-x_2$. 
The generators $\sigma$ and $\overline \sigma$ are represented by $v$ multiplied by explicit polynomials in 
$x_1$ and $x_2$. The bidegrees are $||\sigma ||= (2n-2,1)$ and $||\overline \sigma ||=(2n,1)$.

The cohomology class $\pi_2^*(x_1)$, which equals $\pi_2^*(x_2)$ by Proposition \ref{pr_i} since $p\mid q$,
represents $x_1$ in the spectral sequence. Since the complementary degree of $x_1$ is zero, the algebra structure
of the spectral sequence gives us that $\pi_2^*(x_1)^{n+1}=0$ for $p\mid (n+1)$ and 
$\pi_2^*(x_1)^n=0$ for $p\nmid (n+1)$. 
By the left square in the diagram above, we get that the cohomology class $\pi_1^*(u)$ represents $w$ in the 
spectral sequence.

For $p\mid (n+1)$, $\sigma$ defines uniquely an unfiltered cohomology class since we have $E_\infty^{2n-1,0}=0$.
This class has $\sigma^2=0$ in the filtered quotient but since $E_\infty^{4n-3,1}=E_\infty^{4n-2,0}=0$ this is
also true in the actual cohomology ring.
Similarly, for $p\nmid (n+1)$, $\overline \sigma$ defines uniquely an unfiltered cohomology class with 
$\overline \sigma^2=0$ since $E_\infty^{2n+1,0}=0$ and 
$E_\infty^{4n+1,1} = E_\infty^{4n+2,0}=0$.

Thus for $p\mid (n+1)$ we have a homomorphism of graded rings as follows:
\begin{align*}
& \FF_p [u,x,\sigma]/ (x^{n+1},\sigma^2 ) \to H_\TT^* (\PSt 2 {n+1} q ; \FF_p ) ; \\
& u\mapsto \pi_1^*(u), \quad x \mapsto \pi_2^* (x_1)=\pi_2^*(x_2), \quad \sigma \mapsto \sigma
\end{align*}
The homomorphism induces an isomorphism on associated graded objects, and therefore it is an isomorphism of rings.
By this isomorphism and Proposition \ref{pr_i} we have that 
$pr_i^* (1\otimes x) = \pi_2^* (x_1)= \pi_2^* (x_2)=x$ and 
$pr_i^* (u\otimes 1) = \pi_1^*(u) = u$ as desired. Similarly for $p\nmid (n+1)$.
\end{proof}

\begin{theorem}  \label{Chern_mu}
Let $n$, $q$ and $r$ be integers with $n>1$ and $q>0$. Let $p$ be a prime. Assume that $r=q$ mod $2$. 
Define two polynomials
\begin{align*}
& P(x_1,x_2)=(1+\frac {r+q} {2q} (x_1-x_2))(1+\frac {r-q} {2q}(x_1-x_2)), \\
& R(u)=(1+\frac {r+q} {2} u)(1+\frac {r-q} {2} u).
\end{align*}
In mod $p$ cohomology, we have total Chern classes as follows: If $p\nmid q$,
\begin{equation*}
c((\nu_{r,q})_{h\TT}) = \frac {(1+\frac {r+q} {2q} (x_1-x_2)-x_1)^{n+1}}  {P(x_1,x_2)},
\end{equation*}
\begin{equation*}
c((\overline \nu_{r,q})_{h\TT})=\frac {(1+\frac {r+q} {2q} (x_1-x_2)+x_2)^{n+1}} {P(x_1,x_2)}
\end{equation*}
and if $p\mid q$,
\begin{equation*}
c((\nu_{r,q})_{h\TT}) =\frac {(1+\frac {r+q} {2} u-x)^{n+1}} {R(u)} , \quad
 c((\overline \nu_{r,q})_{h\TT})=\frac {(1+\frac {r+q} {2} u+x)^{n+1}} {R(u)}.
\end{equation*}
\end{theorem}

\begin{proof}
Put $s_i = \frac 1 2 (r+(-1)^{i+1}q)$ for $i=1,2$.
By Proposition \ref{BasicChern} and Remark \ref{BasicChern1} we have that 
\[ 
c_1(\overline \gamma_1 (s_{i})_{h\TT}) =
s_{i} u\otimes 1 -1\otimes x, \quad c_1(\epsilon^1(s_{i})_{h\TT}) = s_{i}u\otimes 1.
\]
Assume that $p\nmid q$. From the pullbacks in Proposition \ref{lbpb} and from Theorem \ref{pr_i1} we get first
Chern classes
\[ c_1\big( ((L_0)_{r,q})_{h\TT} \big) = \frac {s_{i}} q (x_1-x_2)-x_i, \quad
 c_1\big( ((L_i)_{r,q})_{h\TT} \big) = \frac {s_{i}} q  (x_1-x_2).
\]
Note that since $s_1/q (x_1-x_2)-x_1=s_2/q(x_1-x_2)-x_2$ there is no contradiction in the first equation.
By the direct sum decomposition in Theorem \ref{sumiso}, the formula for the total Chern class of $(\nu_{r,q})_{h\TT}$
follows. By a similar argument, we get the formula for the total Chern class of $(\overline \nu_{r,q})_{h\TT}$.

Assume that $p\mid q$. In this case Proposition \ref{lbpb} and Theorem \ref{pr_i1} gives us first Chern classes
$s_i u-x$ and $s_iu$ respectively, and via Theorem \ref{sumiso}, the formula for the total Chern class of
$(\nu_{r,q})_{h\TT}$ follows. Similarly for $(\overline \nu_{r,q})_{h\TT}$.
\end{proof}
We can now prove our second main result regarding the bundles $\mu_q^- \to \B q n$.
\begin{theorem} \label{main2}
Let $n$ and $q$ be integers with $n>1$ and $q>0$. Let $p$ be a prime. In cohomology with mod $p$ coefficients, 
we have total Chern classes as follows: For $p\nmid q$,
\begin{align*}
 c((\mu_q^-)_{h\TT}) = \prod_ {0<s<q}  &(1+\frac s q (x_ 1-x_2)) \cdot \\
\prod_{\substack{0<r<q \\ r = q \text{ mod } 2}} 
& \frac {\big( (1+\frac {r+q} {2q} (x_1-x_2)-x_1)(1+\frac {r+q} {2q} (x_1-x_2)+x_2)\big)^{n+1}}
{\big( (1+\frac {r+q} {2q} (x_1-x_2))(1+\frac {r-q} {2q}(x_1-x_2))\big)^2}. 
\end{align*} 
For $p\mid q$,
\begin{equation*}
 c((\mu_q^-)_{h\TT}) = \prod_ {0<s<q} (1+su) 
\prod_{\substack{0<r<q \\ r = q \text{ mod } 2}} 
\frac {\big( (1+\frac {r+q} {2} u-x)(1+\frac {r+q} {2} u+x)\big)^{n+1}}
{\big( (1+\frac {r+q} {2} u)(1+\frac {r-q} {2}u) \big)^2}. 
\end{equation*} 
\end{theorem}

\begin{proof}
We use the direct sum decomposition from Theorem \ref{main1} which also gives a 
direct sum decomposition after forming $\TT$-homotopy orbit bundles according to Proposition \ref{Borel_vb_prop}. 

The bundle $\epsilon_q (\RR )_{h\TT}$ is trivial so its Chern classes are zero. 
The $\TT$-vector bundle $\epsilon_q (\CC (s))$ is the pullback of $\CP^n \times \CC (s) \to \CP^n$ 
along $pr_i: \PSt 2 {n+1} q \to \CP^n$ both for $i=1$ and $i=2$. 
So by Proposition \ref{BasicChern} and Theorem \ref{pr_i1} we have
\[ c_1(\epsilon_q (\CC (s))_{h\TT} )=pr_i^* (su\otimes 1) = \begin{cases}
\frac s q (x_1-x_2), & p \nmid q, \\
su, & p\mid q.
\end{cases} \]
Theorem \ref{Chern_mu} above gives us the Chern classes of the remaining summands.
\end{proof}


\begin{thebibliography}{MMMM}

\bibitem{Atiyah} Atiyah M. F., $K$-theory, 
Lecture notes by D. W. Anderson, W. A. Benjamin, Inc., New York-Amsterdam 1967.


\bibitem{BHM} B\" okstedt M., Hsiang W.C., Madsen I., 
The cyclotomic trace and algebraic $K$-theory of spaces, Invent. Math. 111 (1993), 465--539.



 

\bibitem{BO5} B\"o{}kstedt M., Ottosen I., String cohomology groups of complex
projective spaces, Algebraic \& Geometric Topology 7 (2007), 2165-2238.



\bibitem{tomD1} tom Dieck T., Transformation groups, de Gruyter 
Studies in Mathematics 8, Walter de Gruyter \& Co., Berlin 1987. 



\bibitem{GHL} Gallot S., Hulin D., Lafontaine J., Riemannian geometry, 
Second edition, Universitext, Springer-Verlag, Berlin 1990. 


\bibitem{Hatcher} Hatcher A., Vector Bundles and $K$-Theory, \newline
http://www.math.cornell.edu/\textasciitilde hatcher

\bibitem{Husemoller} Husemoller D., Fibre bundles, Third edition, 
Graduate Texts in Mathematics 20, Springer-Verlag, New York 1994.



\bibitem{Klingenberg} Klingenberg W., Lectures on closed geodesics.
Grundlehren der Math. Wiss. vol. 230, Springer Verlag, Berlin Heidelberg
New York 1978.

\bibitem{KlingProj} Klingenberg W., 
The space of closed curves on a projective space,
Quart. J. Math. Oxford Ser(2) 20 (1969), 11--31.

\bibitem{KlingSphere} Klingenberg W., 
The space of closed curves on the sphere, Topology 7 (1968), 395--415.

\bibitem{KN2} Kobayashi S., Nomizu K., Foundations of Differential
Geometry, Volume 2 (1969), Wiley and Sons.




\bibitem{MT} Madsen I., Tornehave J., From Calculus to Cohomology,
de Rham cohomology and characteristic classes,
Cambridge university press, Cambridge 1997.





\bibitem{NEH} Ndombol B., El Haouari M., The free loop space equivariant cohomology algebra of 
some formal spaces, Math. Z. 266 (2010), 863-875.







\bibitem{Ziller} Ziller W., The free loop space of globally 
symmetric spaces, Invent. Math. 41 (1977), 1--22.

\end{thebibliography}
\end{document}